\documentclass{amsart}

\usepackage{amsmath}
\usepackage{amssymb}
\usepackage{amsthm}
\usepackage{amsfonts}

\usepackage[a4paper]{geometry}
\usepackage{color}

\usepackage{hyperref}
\hypersetup{
    colorlinks=true,
    linkcolor=blue,
    citecolor=blue,
    }

\usepackage[utf8]{inputenc}
\newtheorem{theorem}{Theorem}
\newtheorem{proposition}{Proposition}
\newtheorem{lemma}{Lemma}
\newtheorem{corollary}{Corollary}
\theoremstyle{definition}
\newtheorem{remark}{Remark}
\newtheorem{example}{Example}

\DeclareMathOperator{\diag}{diag}

\author[O.Ostrovska]{Olha Ostrovska}
\address{National Technical University of Ukraine ``Igor Sikorsky Kyiv
Polytechnic Institute''}
\email{olyushka.ostrovska@gmail.com}

\author[V.Ostrovskyi]{Vasyl Ostrovskyi}
\address{Institute of Mathematics, NAS of Ukraine}
\email{vo@imath.kiev.ua}

\author[L.Turowska]{Lyudmila Turowska}
\address{Chalmers University of Technology, Department of Mathematics}
\email{turowska@chalmers.se}

\title{On quantum symmetries of graphs}
\date{}
\dedicatory{To memory of Yu.M.Berezansky}
\subjclass[2020]{46L67, 47L90, 81P45}
\keywords{Quantum graph, Quantum automorphism, Non-local game}

\begin{document}
\begin{abstract}
Let $G$ be a simple finite graph, and let $\mathcal U_G$ be the related quantum graph. We study the game algebra $C(\mathrm{Qut}(\mathcal U_G))$ of  quantum  automorphism of $\mathcal U_G$. We show that for the complete graph $K_n$, the algebra $C(\mathrm{Qut}(\mathcal U_{K_n}))$  is non-commutative already for all $n\geq 3$, in contrast to $C(\mathrm{Qut}(K_n))=C(S_n^+)$. 
Moreover, we prove that for any graph $G$ with $|V(G)|\geq 3$, the quantum graph $\mathcal U_G$ admits nonlocal symmetry, meaning that there exists a perfect quantum no-signaling correlation for the quantum automorphism game for $\mathcal U_G$ which is not local. 
\end{abstract}

\maketitle

\tableofcontents
\section{Introduction and preliminaries}
 
Let $G$ be a finite simple graph, meaning that it is undirected and has no multiple edges or loops. Write $V(G)$ for its set of vertices and $E(G)$ for the edge set, which is a set of unordered pairs of its vertices. For $x$, $y\in V(G)$ write $x\sim_G y$ or simply $x\sim y$, if $G$ is clear from the context, if there is an edge between $x$ and $y$ and let $A_G=(a_{x,y})_{x,y\in V(G)}$, be the adjacency matrix of $G$: we have $a_{x,y}=1$ if $x\sim_G y$, and $a_{x,y}=0$ otherwise. 
Let $G$ and $H$ be two simple graphs. We say that $G$ and $H$ are isomorphic if there exists a bijection $\varphi:V(G)\to V(H)$ such that $\varphi(x)\sim_H\varphi(y)$ if and only if $x\sim_G y$.  Clearly, in this case $|V(G)|=|V(H)|$. The map $\varphi$ is called an isomorphism of graphs. An automorphism $\varphi$ of a graph $G$ is an isomorphism from $G$ to itself. The automorphisms build a group with respect to composition known as the automorphism group of $G$.

Graph-based non-local games, and, in particular, isomorphism games have been important in the process of understanding entanglement in quantum mechanics, and first studied in \cite{amrssv}.   In a (2-players) non-local game, a trusted classical verifier interacts with two players, Alice and Bob, who collaborate in attempt to win the game by sending a question/input to each of the two players and they must respond with an answer/output. In the case of isomorphism games for simple graphs $G$ and $H$ with $|V(G)|=|V(H)|$ the verifier sends two vertices $x$ and $y$ of the graph $G$ and Alice and Bob must return vertices $a$ and $b$ of the graph $H$. In order to win, $\text{rel}(x,y)=\text{rel}(a, b)$ must hold where $\text{rel}$ is the relationship of two vertices that is the equality of the vertices or adjacency relation or non of the previous two.  A strategy for $(G,H)$-isomorphism is given by joint conditional probability distributions, called also correlations,  $p(a,b|x,y)$, where $x,y \in V(G)$  and $a,b\in V(H)$ and where $p(a,b|x,y)$ is thought as the probability for Alice and Bob give the answer $(a,b)$ upon receiving the questions $(x,y)$. We call the strategy is winning or perfect if $p(a,b|x,y)=0$, whenever $\text{rel}(x,y)\ne \text{rel}(a,b)$.  There are several subclasses of strategies that arise from different models of quantum mechanics, related to the way the players compose their individual physical systems. A deterministic classical strategy is given by 01-valued correlations: the Alice's answer is completely determined by the question she received and similar for Bob; in particular, the strategy is given by two functions $\varphi_A:V(G)\to V(H)$, $\varphi_B:V(G)\to V(H)$. For such strategy to be perfect it is easy to see that $\varphi_A=\varphi_B$, where the latter would be the actual isomorphism between $G$ and $H$. In general, classical (local)  strategies, one allows Alice and Bob to pick a deterministic strategy at random. The existence of a perfect local strategy gives the isomorphism of the graphs. A quantum commuting strategy is given by a Hilbert space $\mathcal H$, a state $\xi$, interpreted as a unit vector in $\mathcal H$, on which local measurements are
jointly performed and mutually commuting families of  measurements (POVMs) $E_x=(E_{x,a})_{a\in V(H)}$, $x\in V(G)$,  for Alice and $F_y=(F_{y,b})_{b\in V(H)}$, $y\in V(G)$, for Bob, that are used to measure $\xi$ with the resulting outcomes given by 
$$p(a,b|x,y)=\langle E_{x,a}F_{y,b}\xi,\xi\rangle. $$
The strategy is called quantum if $\mathcal H$ is a tensor product of finite-dimensional Hilbert spaces $\mathcal H_A\otimes \mathcal H_B$ and the product $E_{x,a}F_{y,b}$ is replaced by the tensor products of the POVMs defined now on the finite-dimensional spaces. Note that a quantum strategy can be obtained as a quantum commuting one by considering the POVMs $E_{x,a}\otimes 1_{\mathcal H_B}$ and $1_{\mathcal H_A}\otimes F_{y,b}$. Here a POVM on a Hilbert space $\mathcal H$ is a family of positive operators on $\mathcal H$ which sum up to the identity operator.  Local, quantum and quantum commuting strategies are examples of so called no-signalling correlations (NS correlations for short). The isomorphism game is a bisynchronous game, implying that the winning strategy satisfies $p(a,b|x,x)=p(a,a|x,y)=0$ for all $x\ne y$, $a\ne b$, and by \cite{paulsen-rahaman} the winning strategy of the quantum commuting (qc) type is given by a trace on the associated game algebra. Namely, for the isomorphism game $\text{Iso}(G,H)$,
the game algebra is the universal $C^*$-algebra defined by generators and relations as follows:  
\begin{align*}
\mathcal A(\text{Iso}(G,H))=C^*\Bigl\langle&
e_{x,a}, x\in V(G), a\in V(H): e_{x,a}^2=e_{x,a}=e_{x,a}^*,
\\ &\sum_x e_{x,a}=\sum_ae_{x,a}=1, \quad
(A_G\otimes 1)u=u(A_H\otimes 1)\Bigr\rangle,
\end{align*}
where $u=(e_{x,a})_{x\in V(G),a\in V(H)}$. Any winning qc-strategy is then given by 
$$p(a,b|x,y)=\tau(e_{x,a}e_{b,y}),$$
where $\tau$ is a tracial state on $\mathcal A(\text{Iso}(G,H))$. 
It is a local (quantum) strategy if $\tau$ is factorized through an abelian (finite-dimensional) $C^*$-algebra $\mathcal A$, that is $\tau=\tau_{\mathcal A}\circ\pi$, where $\pi: \mathcal A(\text{Iso}(G,H))\to\mathcal A$ is a $*$-homomorphism and $\tau_{\mathcal A}$ is a state (trace) on $\mathcal A$. 

In what follows, we will write 
$C(\mathrm{Qut}(G))$ for $\mathcal A(\text{Iso}(G,G))$, the quantum automorphism group; the notion is due to Banica, 2005
\cite{banica}.
One says that $G$ has quantum symmetry if $C(\mathrm{Qut}(G))$ is not commutative.

If $K_N$ is the complete graph on $N$ vertices, $C(\mathrm{Qut}(K_N))=C(S_N^+)$, called the free quantum permutation group, introduced by Wang \cite{wang} and studied intensively from various perspectives (see \cite{bbc, lmr, moritz, voigt} and references therein). It is the universal $C^*$-algebra generated by entries of $N\times N$ magic unitary, that is projections
$p_{i,j}$, $i,j=1,\ldots N$, such that $(p_{i,j})_{i,j=1}^N$, is unitary, equivalently, $\sum_i p_{i,j}=\sum_j p_{i,j}=1$. 
It is known that $C(S_N^+)$ is not commutative iff $N\geq 4$, see \cite{voigt}. 

While classically, in the course of a non-local game, 
 Alice and Bob's behaviour
is described by a family $p = (p(a,b|x,y))_{(a,b,x,y) \in A \times B \times X \times Y}$ of conditional 
probability distributions, which can, in a canonical way, be viewed as a noisy information channel 
$\mathcal N : X \times Y \to A \times B$ with well-defined marginal channels, 
in the quantum setting, one replaces the classical state spaces $X, Y, A, B$ by their quantum analogues (i.e. the Hilbert spaces $\mathbb C^{|X|}, \mathbb C^{|Y|},$ etc.), and the classical channel $\mathcal N : X \times Y \to A \times B$ 
by a {\it quantum channel} $\Gamma: M_X \otimes M_Y \to M_A \otimes M_B$, where, for any finite set $Z$, 
we have let $M_Z =  B(\mathbb C^{|Z|})$ be the matrix algebra of linear maps on $\mathbb C^{|Z|}$.  
Quantum no-signalling (QNS) correlations have been introduced by Duan and Winter in \cite{duan-winter}, and their different classes were introduced and studied in \cite{tt-QNS}.  An analogue of bisynchronous correlations, called QNS-bicorrelations,  was proposed in \cite{bhtt}, as no-signalling QNS correlations satisfying certain concurrency conditions; they are characterised in terms of tracial states on the $C^*$-algebra  $C({\mathbb P}{\mathcal  U}_X^+)$ 
of functions on {\it projective free unitary quantum group}.  
Recall that the $C^*$-algebra of the free unitary quantum group $C(\mathcal U_X^+)$ is the universal unital $C^*$-algebra generated by the entries $u_{x,a}$ of an $|X| \times |X|$ bi-unitary matrix $U = (u_{x,a})_{x,a}$ which was introduced by Wang, \cite{wang}, and studied by Banica in \cite{banica1}, 
and the $C^*$-algebra $C({\mathbb P}{\mathcal U}_{X}^+)$
is the $C^*$-subalgebra of
$C({\mathcal  U}_X^+)$,
generated by length two words of the form $u_{x,a}^*u_{x',a'}$, $x,x',a,a'\in X$, which is the quantum automorphism group $\text{Aut}^+(M_X)$ of Wang due to a result by Banica, \cite{banica2}. 

In \cite{bhtt} bisynchronous  bicorrelations were used to study quantum graph isomorphism games, 
where by a quantum graph we mean a symmetric skew subspace ${\mathcal U}\subset \mathbb C^X\times \mathbb C^X$ (see \cite[Definition 7.1]{bhtt}).  To a classical graph $G$ one can associate naturally the quantum graph  $\mathcal U_G=\text{span}\{e_x\otimes e_y: x\sim y\}$, where $\{e_x\}_{x\in X}$ is the natural orthogonal basis in $\mathbb C^X$, "embedding"  in this way classical graphs into the class of quantum graphs. Similar to the classical case one defines quantum isomorphisms between quantum graphs in terms of perfect 
QNS strategies for a suitable quantum graph isomorphism game. We refer the reader to \cite{bhtt} for the details.    
The game algebra whose tracial states encode such perfect strategies is a
quotient of $C({\mathbb P}{\mathcal  U}_X^+)$. For the classical graphs $G$,
$H$, $\mathcal A(\text{Iso}(\mathcal
U_G,\mathcal U_H))$ 
is the  $C^*$-subalgebra generated by $u_{a,x}^*u_{b,y}$, $a,b,x,y\in X=V(G)=V(H)$ of the universal $C^*$-algebra with generators $u_{a,x}$, $a,x\in X$, and relations such that $U = (u_{x,a})_{x,a\in X}$ is a bi-unitary and 
 $u_{a,x}u_{b,y}^*=0$ whenever $x\sim y$
and $a\not\sim b$ or $x\not\sim y$ and $a\sim b$.  Note, that by some abuse of
terminology, in the paper we will always refer to this algebra as the universal
$C^*$-algebra generated by the words  $u_{a,x}^*u_{b,y}$, $a,b,x,y\in X$, where
$U = (u_{x,a})_{x,a\in X}$ is a bi-unitary, whose entries satisfy the given
additional relations.
It was proved in \cite{bhtt} that there exists a local perfect QNS strategy for
the quantum $(\mathcal U_G,\mathcal U_H)$-isomorphism game if and only if $G$
and $H$ are isomorphic, and hence there is a local strategy for
$(G,H)$-isomorphism game.
However, it is unknown whether genuinely quantum (quantum/quantum commuting
QNS) strategies will have an advantage over quantum/quantum commuting NS
correlations  to play the $(G,H)$-isomorphism game\footnote{After the paper was
completed, D. Roberson informed the third author, in private communication, that
there exist classical graphs $G$ and $H$  such that $\mathcal U_G$ and $\mathcal
U_H$ are quantum isomorphic yet there is no winning quantum no-signalling
correlation for the isomorphism game $\text{Iso}(G,H)$.}.  Treating classical
graphs as quantum, one may ask whether they have a more rich quantum symmetries
and this is the goal of our paper. Namely, in this work we study the algebra
$C(\mathrm{Qut}(\mathcal U_G)):=\mathcal A(\text{Iso}(\mathcal
U_G,\mathcal U_G))$ of the quantum  automorphism game of $\mathcal U_G$. We show that for the complete graph $K_n$, the algebra $C(\mathrm{Qut}(\mathcal U_{K_n}))$  is non-commutative already for all $n\geq 3$, in contrast to $C(\mathrm{Qut}(K_n))=C(S_n^+)$.   Moreover, there is a surjective homomorphism from  $C(\mathrm{Qut}(\mathcal U_{K_n}))$ to $C^*({\mathbb F}_{n-1})$, the $C^*$-algebra of the free group on $n-1$ generators. In section \ref{nonloc} we prove that for any graph $G$ with $|V(G)|\geq 3$, the corresponding quantum graph $\mathcal U_G$ admits nonlocal symmetry,
meaning that there exists a perfect QNS correlation for the quantum automorphism game for $\mathcal U_G$ which is not local. The similar question for classical graphs was studied by Roberson and Schmidt in \cite{RobSch21}, where they proved among other results, that the complete graph $K_n$ exhibit nonlocal symmetry if and only if $n\geq 5$.   In section \ref{dec}  we also present a partition procedure that allows to refine the description of  quantum isomorphisms of quantum graphs of type $\mathcal U_G$.  

For a Hilbert space $H$ we write $B(H)$ for the space of all bounded operators. As usual, $M_n(\mathbb C)$ or simply $M_n$ is the space of $n\times n$ complex matrices. If $A$ is a $C^*$-algebra $M_n(A)$ is the $C^*$-algebra of matrices with entries in $A$. We write $\mathbb N_n$ for the set $\{1,\ldots, n\}$. 

\subsection*{Acknowledgements}
O.O. and V.O. thank Chalmers University of Technology which supported their visit to Gothenburg in 2023, 2024. L.T. thanks GS Magnusons Fond MF2022-0006, which supported the visit of O.O. and V.O. to Gothenburg in 2023, 2024. V.O. thanks Simons Foundation for financial support (grant SFI-PD-Ukraine-00014586, VO).
This material was also based upon work supported by the Swedish Research Council under grant no.2021-06594 while L.T. was in residence at Institut Mittag-Leffler in Djursholm during the semester "Operator algebras and Quantum information theory" Spring, 2026.

\section{Game $C^*$-algebra for complete graphs}

\subsection{Structure of the game algebra $C(\mathrm{Qut}(\mathcal U_{K_n}))$}
Recall that if $\mathcal A$ is a $C^*$-algebra, a block matrix $U=(u_{i,j})_{i,j=1}^n\in M_n(\mathcal A)$ is called a bi-unitary if $U$ and $U^t=(u_{j,i})_{i,j=1}^n$ are unitary.

Let $K_n$ be the complete graph on $n$ vertices, $n\geq 1$.  By definition, the algebra $C(\mathrm{Qut}(\mathcal U_{K_n}))$
is the 
universal $C^*$-algebra generated by $u_{i,j}^*u_{k,l}$, $i,j,k,l=1,\ldots, n$,
where $U=(u_{i,j})_{i,j=1}^n$ is a bi-unitary, such that $u_{i,k}u_{i,l}^*=0$
for $k\ne l$ and $u_{i,k}u_{j,k}^*=0$, when $i\ne j$.
The next theorem gives an embedding of $C(\mathrm{Qut}(K_n))$ into $C(\mathrm{Qut}(\mathcal U_{K_n}))$. 

\begin{theorem}\label{prop:aiso} For each $n\in\mathbb N$,
$C(\mathrm{Qut}(\mathcal U_{K_n}))$ is the universal
$C^*$-algebra generated by $u_{i,j}^*u_{k,l}$, $i,j,k,l\in \mathbb N_n$ where each
$u_{i,j}$ is a partial isometry and  $U=(u_{i,j})_{i,j=1}^n$ is a  bi-unitary. The
map $\varphi:C(S_n^+)\to C(\mathrm{Qut}(\mathcal U_{K_n}))$, $p_{i,k}\to u_{i,k}^*u_{i,k}$, $i,k\in \mathbb N_n$, is an injective $*$-homomorphism.
\end{theorem}

\begin{proof}
If $u_{k,l}^*u_{i,j}$, $i,j,k,,l\in\mathbb N_n$, are the generators of $C(\mathrm{Qut}(\mathcal
U_{K_n}))$, then $u_{i,j}u_{i,k}^*=0$ if $j\ne k$. Therefore,
for all $i$, $k$,
\[
u_{i,k}^*=\sum_ju_{i,j}^*u_{i,j}u_{i,k}^*=\sum_{j\ne
k}u_{i,j}^*u_{i,j}u_{i,k}^*+u_{i,k}^*u_{i,k}u_{i,k}^*=u_{i,k}^*u_{i,k}u_{i,k}^*,
\]
i.e. each $u_{i,k}$ is a partial isometry. Conversely, if $U=(u_{i,j})_{i,j=1}^n$ is
a bi-unitary such that each $u_{i,j}$ is a partial isometry we have
$u_{i,k}u_{j,k}^*=0$ and $u_{k,i}u_{k,j}^*=0$ if $i\ne j$. In fact, in this case
each $u_{i,k}^*u_{i,k}$ is a projection and $\sum_i u_{i,k}^*u_{i,k}=1$ showing
that the projections are orthogonal and
$$u_{i,k}u_{j,k}^*=u_{i,k}u_{i,k}^*u_{i,k}u_{j,k}^*u_{j,k}u_{j,k}^*=0, \quad i\ne j.$$

Similarly, $u_{k,i}u_{k,j}^*=0$ if $i\ne j$. Therefore we obtain the first
statement. Remark also that
 as $\sum_ku_{i,k}^*u_{j,k}=0$ for $i\ne j$,
\[
u_{i,m}u_{i,m}^*u_{j,m}=u_{i,m}\sum_k u_{i,k}^*u_{j,k}=0,
\]
showing that also $u_{i,m}^*u_{j,m}=0$ if $i\ne j$, and similarly,
$u_{m,i}^*u_{m,j}=0$.

Clearly, the map $\varphi$ is a $*$-homomorphism. It is injective. In fact,
letting $\mathcal U_{n,p.i.}^+$ be the universal $C^*$-algebra generated by
partial isometries $u_{i,j}$ such that $U=(u_{i,j})_{i,j=1}^n$ is a bi-unitary and
considering $\rho:\mathcal U_{n,p.i.}^+\to C(S_n^+)$, $u_{i,j}\mapsto p_{i,j}$,
we obtain $\rho(\varphi(a))=a$, $a\in C(S_n^+)$.
\end{proof}

\subsection{Properties of $C(\mathrm{Qut}(\mathcal U_{K_2}))$ and $C(\mathrm{Qut}(\mathcal U_{K_3}))$}
The $C^*$-algebra $C(S_2^+)$ is generated by projections $p_{1,1}$, $p_{1,2}$, $p_{2,1}$,
$p_{2,2}$ such that $p_{1,1}+p_{1,2}=1$, $p_{2,1}+p_{2,2}=1$,
$p_{1,1}+p_{2,1}=1$, $p_{1,2}+p_{2,2}=1$ showing that $p_{1,2}=p_{2,1}$,
$p_{1,1}=p_{2,2}$ and hence $C(S_2^+)$ is the universal $C^*$-algebra generated by one
projection.
\begin{proposition}\label{graph_K2}
$C(\mathrm{Qut}(\mathcal U_{K_2}))$ is commutative and
$C(\mathrm{Qut}(\mathcal U_{K_2}))\simeq C(\mathbb
T)\oplus C(\mathbb T)$.
\end{proposition}

\begin{proof}
Let $u_{i,j}$, $i,j=1,2$, be partial isometries  such that $u_{i,j}^*u_{k,l}$, $i,j,k,l=1,2$,
generate $C(\mathrm{Qut}(\mathcal U_{K_2}))$.  The only
possibly non-zero products are $u_{1,1}^*u_{1,1}=u_{2,2}^*u_{2,2}$,
$u_{1,1}^*u_{2,2}$, $u_{2,2}^*u_{1,1}$, $u_{1,2}^*u_{2,1}$, $u_{2,1}^*u_{1,2}$,
$u_{1,2}^*u_{1,2}=1-u_{1,1}^*u_{1,1}=u_{2,1}^*u_{2,1}$. It is a routine to check
that they commute.

Let $\varphi:C( \mathcal U_2^+)\to C(\mathbb T)\oplus C(\mathbb T)$,
\[
\varphi(u_{1,1})=(1,0),\quad \varphi(u_{2,2})=(z,0),\quad \varphi(u_{1,2})=(0,1),\quad
\varphi(u_{2,1})=(0,z).
\]
It is easy to see that $\varphi$ determines a surjective $*$-homomorphism of 
of $C(\mathrm{Qut}(\mathcal U_{K_2}))$ onto $C(\mathbb T)\oplus
C(\mathbb T)$.
As any one-dimensional  representation of  $C(\mathrm{Qut}(\mathcal
U_{K_2}))$ is given by
$\pi_z(u_{1,1}^*u_{1,1})=1$, $\pi_z(u_{1,1}^*u_{2,2})=z$, and
$\pi_z(u_{1,2}^*u_{1,2})=\pi_z(u_{2,1}^*u_{2,1})=0$ or
$\rho_z(u_{1,2}^*u_{1,2})=1$, $\rho_z(u_{1,2}^*u_{2,1})=z$, and
$\rho_z(u_{1,1}^*u_{1,1})=\rho_z(u_{2,2}^*u_{2,2})=0$, $z\in\mathbb T$,  we
obtain that the $*$-homomorphism is a $*$-isomorphism.
\end{proof}

As $C(S_n^+)$ is non-commutative for $n\geq 4$ we have already that $C(\mathrm{Qut}(\mathcal U_{K_n}))$ is non-commutative for the same values of $n$, as
$C(S_n^+)$ sits in the latter algebra.

We will see that $C(\mathrm{Qut}(\mathcal U_{K_3}))$ is not commutative. But first we will establish some structural properties of the algebra $C(\mathrm{Qut}(\mathcal U_{K_n}))$. 

\begin{lemma}
 Let $U = (u_{a,x})_{a,x=1}^n$ be a bi-unitary matrix which determines a
representation of $C(\mathrm{Qut}(\mathcal U_{K_n}))$. There
exists a block-diagonal unitary matrix $V = \diag(v_1,\dots, v_n)$ such that
$VU$ is a projection permutation matrix, i.e., $(v_a u_{a,x})_{a,x=1}^n$ is a
family of projections which forms a representation of $C(S_n^+)$.
\end{lemma}

\begin{proof}
 Put
 \[
  v_a = \sum_{x=1}^n u^*_{a,x}, \quad a =1,\dots, n.
 \]
 We claim that $v_a$, $a\in\mathbb N_n$, form the desired family. Indeed, as shown in the proof of
Theorem~\ref{prop:aiso}, $u_{a,x}u^*_{a,y} = u^*_{a,x}u_{a,y} =0$, $x \not\sim
y$, then we have
\begin{align*}
 v_a^*v_a& = \sum_{x,y=1}^n u_{a,x}u^*_{a,y} = \sum_{x=1}^n u_{a,x}u^*_{a,x} =I,
 \\
 v_av_a^*& = \sum_{x,y=1}^n u^*_{a,x}u_{a,y} = \sum_{x=1}^n u^*_{a,x}u_{a,x} =I,
\end{align*}
therefore, $v_a$ is unitary. Further, we have
\[
 v_a u_{a,x} = \sum_{y=1}^n u^*_{a,y} u_{a,x} = u^*_{a,x} u_{a,x} =: p_{a,x},
\]
which, as shown above, is a projection, and moreover, these projections form a
representation of $C(S_n^+)$.
\end{proof}

\begin{theorem}\label{generator}
  The algebra $C(\mathrm{Qut}(\mathcal U_{K_n}))$ is
isomorphic to the universal $C^*$-algebra  $\mathcal  A$ generated by
projections $p_{a,x}$, $a, x\in \mathbb N_n$, and unitaries $u_i$, $i\in\mathbb N_{n-1}$, such that $\sum_{a=1}^n p_{a,x}=\sum_{x=1}^n p_{a,x}=1$ and $p_{b,x}u_b\ldots
u_{a-1}p_{a,x}=0$ for $b<a$.
  \end{theorem}

\begin{proof}
    Let $v_a=\sum_{x=1}^n u^*_{a,x}$, $a\in\mathbb N_n$ and set $\tilde
p_{a,x}=v_au_{a,x}$, $\tilde u_a=v_av_{a+1}^*$, $a\in\mathbb N_{n-1}$. Then $\tilde
p_{a,x}$ and $\tilde u_a\in C(\mathrm{Qut}(\mathcal U_{K_n}))$. We have  $u_{a,x}^*u_{b,y}=\tilde p_{a,x}v_av_b^*\tilde p_{b,y}$,
$v_av_b^*=\tilde u_a\tilde u_{a+1}\ldots \tilde u_{b-1}$, if $b>a$, and
$v_av_b^*=(v_bv_a^*)^*$, otherwise, showing that $\tilde p_{a,x}$  and $\tilde
u_a$ generate $C(\mathrm{Qut}(\mathcal U_{K_n}))$. Moreover,
from the proof of the previous lemma we get that $\tilde u_a$'s are unitary,
$\tilde p=(\tilde p_{a,x})_{a,x=1}^n$ is a projection permutation matrix and
    \[\tilde p_{b,x}\tilde u_b\ldots \tilde u_{a-1}\tilde
p_{a,x}=u_{b,x}^*u_{a,x}=0
    \] for $b<a$, as follows from the proof of Theorem \ref{prop:aiso}. Therefore, we get a  $*$-homomorphism $\varphi$ from $C(\mathrm{Qut}(\mathcal U_{K_n}))$ to $\mathcal A$, given by $\varphi(\tilde
p_{a,x})=p_{a,x}$, $\varphi(\tilde u_a)=u_a$.

For the generators $u_a$, $a\in\mathbb N_{n-1}$,  of ${\mathcal A}$, let $v_1=1$, $v_a=u_{a-1}^*\ldots
u_1^*$, $a>1$. Clearly,  $v_a$ and $p_{a,x}$ generate $\mathcal A$.
Let $ V=\diag ( v_1,\ldots,  v_a)$. Then $V$ is a unitary diagonal block matrix such that 
$V^*P=(
v_a^*p_{a,x})_{a,x=1}^n$ is a bi-unitary block matrix:
\begin{align*}
 &\sum_{a=1}^np_{a,y}v_av_a^*p_{a,x}=\sum_{a=1}^np_{a,y}p_{a,x}=\delta_{x,y} I\\
 &\sum_{x=1}^nv_a^*p_{a,x}p_{b,x}v_b=v_a^*\Bigl(\sum_{x=1}^np_{a,x}p_{b,x}\Bigr)v_b=\delta_{a,b} I\\
 &\sum_{x=1}^np_{b,x}v_bv_a^*p_{a,x}=\sum_{x=1}^n p_{b,x}u_b\ldots
u_{a-1}p_{a,x}=\delta_{a,b}I,\quad b<a,\\
 &\sum_{a=1}^nv_a^*p_{a,x}p_{a,y}v_a=\delta_{x,y} I.
 \end{align*}
 Moreover, using $\sum_{y=1}^np_{b,y}=1$ for all $b\in \mathbb N_n$, it is easy to see that the elements $p_{b,y}u_b\ldots u_{a-1}p_{a,x}$, $b\leq a$ (where in the case $b=a$ we think of the latter to be equal to $p_{a,y}p_{a,x}$)
generate $\mathcal A$. 
  Therefore, we obtain a $*$-homomorphism $\psi: \mathcal A\to C(\mathrm{Qut}(\mathcal U_{K_n}))$,
  $\psi(p_{b,y}u_b\ldots u_{a-1}p_{a,x})=u_{b,y}^*u_{a,x}$.

  We have $$\psi(\varphi(\tilde
u_a))=\psi(u_a)=\psi\Bigl(\sum_{x,y=1}^np_{a,y}u_ap_{a+1,x}\Bigr)=\sum_{x,y=1}^nu_{a,y}^*u_{a+1,x}=v_av_{a+1}^*
=\tilde u_a,$$
and 
  $$\psi(\varphi(\tilde p_{a,x}))=\psi(p_{a,x})=u_{a,x}^*u_{a,x}=\tilde p_{a,x},$$
i.e. $\psi\circ\varphi=\text{id}$. Similarly,
  we obtain $\varphi\circ\psi=\text{id}$.
\end{proof}
Let $C^*(\mathbb F_n)$ be the $C^*$-algebra of the free group $\mathbb F_n$ with
$n$ generators.  Write $w_i$, $1\leq i\leq n$, for the canonical generators of
$C^*(\mathbb F_n)$.
\begin{corollary}
 The map $\varphi\colon C(\mathrm{Qut}(\mathcal U_{K_n}))\to
C^*(\mathbb F_{n-1})$, $p_{a,x}\mapsto\delta_{a,x}1$, $u_i\mapsto w_i$, $a,x\in\mathbb N_n$, $i\in\mathbb N_{n-1}$, is a
surjective $*$-homomorphism.
\end{corollary}
\begin{proof}
   As $\sum_{x=1}^n\varphi(p_{a,x})=\sum_{a=1}^n\varphi(p_{a,x})=1$ and
\[
\varphi(p_{b,x})\varphi(u_b\ldots u_{a-1})\varphi(p_{a,x})=0
\]
for $b\ne  a$, the
statement now follows from the theorem.
\end{proof}
\begin{corollary}
The game algebra $C(\mathrm{Qut}(\mathcal U_{K_n}))$ is non-commutative for all $n\geq 3$. 
\end{corollary}

Next example presents some other in nature $*$-representations of $C(\mathrm{Qut}(\mathcal U_{K_3}))$ showing also that the category of
$*$-representations of $C(\mathrm{Qut}(\mathcal U_{K_3}))$ is at least  as complicated as the category of 
$*$-representations of the free
product of the Cuntz algebras $\mathcal O_2$.

\begin{example}
 Let $H$ be a Hilbert space, and let isometries $t_1,t_2,s_1,s_2$ in $B(H)$ satisfy the Cuntz algebra relations 
 \[
  t_1t_1^* + t_2 t_2^* = I, \quad s_1s_1^* + s_2s_2^* = I,
 \]
 i.e., determine a representation of the free product $\mathcal O_2*\mathcal
O_2$, amalgamated over the units. 
In the space $\mathcal H = H\oplus H\oplus H$, consider the operators
\begin{align*}
 p_{1,1} = p_{2,3} = p_{3,2} = \begin{pmatrix} I&0&0\\0&0&0\\0&0&0\end{pmatrix},
 \\
 p_{1,2} = p_{2,1} = p_{3,3} = \begin{pmatrix} 0&0&0\\0&I&0\\0&0&0\end{pmatrix},
 \\
 p_{1,3} = p_{2,2} = p_{3,1} = \begin{pmatrix} 0&0&0\\0&0&0\\0&0&I\end{pmatrix},
\end{align*}
and
\[
 u_1 =
 \begin{pmatrix}
 t_1^* & 0&0\\ s_1t_2^* & s_2s_1^* &0 \\ 0& t_2 s_2^* &t_1
 \end{pmatrix},
 \quad
 u_2 =
 \begin{pmatrix}
 t_1 t_2^* & 0&t_2t_1^*\\ s_2t_1^* & s_1 &0 \\ 0& 0 &t_2^*
 \end{pmatrix}.
\]
It is a straightforward verification that $p_{a,x}$, $a,x =1,2,3$, are projections
for which
\[
 \sum_{b=1}^3 p_{b,x} = \sum_{y=1}^3 p_{a,y} =I, \quad a,x=1,2,3,
\]
$u_1$ and $u_2$ are unitary,
and
\[
 p_{1,k}u_1 p_{2,k} = p_{2,k}u_2 p_{3,k} = p_{1,k}u_1u_2 p_{3,k}=0,\quad k=1,2,3,
\]
and therefore determine a $*$-representation of the category of
$*$-representations of $C(\mathrm{Qut}(\mathcal U_{K_3}))$, by Theorem~\ref{generator}.
Moreover, the representation of $C(\mathrm{Qut}(\mathcal U_{K_3}))$ given by $u_1,u_2$, $p_{a,x}$, $a,x =1,2,3$,  is irreducible iff
the corresponding representation of $\mathcal O_2*\mathcal O_2$ is irreducible,
and two such representations of $C(\mathrm{Qut}(\mathcal U_{K_3}))$ are unitarily equivalent iff the corresponding representations of
$\mathcal O_2*\mathcal O_2$ are unitarily equivalent.

Indeed, let $\tilde t_1,\tilde t_2,\tilde s_1,\tilde s_2$  be another
representation of $\mathcal O_2 *\mathcal O_2$ in $\tilde H$, let  $\tilde
p_{a,x}$, $a,x=1,2,3$ be the corresponding projections in $\tilde{ \mathcal H} =
\tilde H \oplus \tilde H \oplus \tilde H$, and let $\tilde u_1,\tilde u_2$ be a pair of
the corresponding block unitaries. Let $C$ be a bounded operator from $\tilde{\mathcal H}$ to $\mathcal H$
intertwining $u_1,u_2$, $p_{a,x}$, $a,x=1,2,3$, and $\tilde u_1,\tilde u_2$, $\tilde
p_{a,x}$, $a,x=1,2,3$. Then $C=\diag(C_1,C_2,C_3)$, $C_1,C_2,C_3\colon \tilde H \to H$, and relations
\[
 u_1C =C\tilde u_1,\quad  u_1^*C = C \tilde u_1^*,\quad u_2C = C \tilde u_2,\quad  u_2^*C = C \tilde u_2^*,
\]
are equivalent to the following set of relations
\begin{gather*}
t_1C_1 =C_1\tilde t_1,\quad t_1^* C_1=C_1 \tilde t_1^*, \quad  t_2C_3 =
C_3\tilde t_2, \quad  t_2^*C_3 = C_3\tilde t_2^*,
\\
t_1C_3=C_3\tilde t_1, \quad t_1^*C_3 = C_3\tilde t_1^*, \quad  s_1 C_2=C_2
\tilde s_1, \quad s_1^* C_2 =C_2 \tilde s_1^* ,
\\
\label{eq:k3-t2t1}
t_2 t_1^* C_1=C_1 \tilde t_2 \tilde t_1^* ,
\quad
t_1 t_2^* C_1=C_1 \tilde t_1 \tilde t_2^*,
\\
\label{eq:k3-t1s2}
t_1s_2^*C_2 = C_1 \tilde t_1\tilde s_2^*,
\quad
s_2t_1^*C_1 = C_2 \tilde s_2\tilde t_1^*,
\\
\label{eq:k3-t1t2}
t_1t_2^*C_1 =C_3 \tilde t_1\tilde t_2^*,
\quad
t_2t_1^*C_3 = C_1\tilde t_2\tilde t_1^*,
\\
\label{eq:k3-s1t2}
s_1t_2^* C_1= C_2\tilde s_1\tilde t_2^*,
\quad
 t_2s_1^*C_2= C_1 \tilde t_2\tilde s_1^*,
\\
\label{eq:k3-s1s2} s_2s_1^* C_2=C_2\tilde s_2\tilde s_1^*,
\quad
s_1s_2^*C_2 =C_2 \tilde s_1\tilde s_2^*,
\\
\label{eq:k3-s2r2}
 s_2t_2^*C_3 = C_2\tilde s_2\tilde t_2^* ,
\quad
t_2 s_2^*C_2=C_3\tilde t_2 \tilde s_2^*,
\end{gather*}
which imply
\begin{gather*}
 t_2C_1 =C_1\tilde t_2, \quad t_2^* C_1 = C_1 \tilde t_2^*, \quad t_2^*C_1 = C_3
\tilde t_2^*, \quad t_2 C_3 = C_1\tilde t_2,
 \\
 s_2^* C_2 = C_1 \tilde s_2^*, \quad  s_2C_1 =C_2 \tilde s_2, \quad s_1C_1 = C_2
\tilde s_1,\quad s_1^*C_2 = C_1 \tilde s_1^*,
 \\
 s_2C_2 = C_2\tilde s_2, \quad  s_2^*C_2 = C_2 \tilde s_2^*, \quad t_2^*C_3 =
C_2 \tilde t_2^*, \quad t_2C_2 = C_3 \tilde t_2,
\end{gather*}
so that $C_1=C_2 = C_3$ intertwines $t_1,t_2,s_1,s_2$ and $\tilde t_1,\tilde
t_2,\tilde s_1, \tilde s_2$. This proves the statements about irreducibility and
unitary equivalence above.
\end{example}

\section{Nonlocal symmetries}

\subsection{Nonlocal symmetries for $K_3$}
We define nonlocal symmetry of a quantum graph in a similar way as it is done
for classical graphs.
We say a quantum graph $\mathcal U$ admits   nonlocal symmetry if there is a
perfect QNS correlations for the automorphism game $\text{Iso}(\mathcal U,\mathcal U)$ of $\mathcal U$, which is not local. We refer the reader to \cite{bhtt} for the  definition of local, quantum and quantum commuting QNS strategies. In what follows we will use the following characterization of perfect QNS correlation for isomorphism game, see \cite[Theorem 7.10]{bhtt}.
\begin{theorem}\label{th_alg}
Let $G$, $H$ be simple undirected graphs such that $|V(G)|=|V(H)|=n$. The following are equivalent for a QNS bicorrelation $\Gamma : M_{n}\otimes M_n\to
M_n\otimes M_n$:
\begin{itemize}
\item[(i)] $\Gamma$ is a perfect quantum commuting (resp. quantum/local)
strategy for the isomorphism game
$\text{Iso}(\mathcal U_{G},\mathcal U_{H})$;
\item[(ii)] there exists a trace $\tau$
(resp. a trace $\tau$ that factors through a finite dimensional/abelian
*-representation)
of $\mathcal A(\text{Iso}(\mathcal U_{G},\mathcal U_{H}))$ such that
\begin{equation}\label{eq_agains}
\Gamma(\epsilon_{x,x'} \otimes \epsilon_{y,y'})
= \left(\tau(u_{a,x}^*u_{a',x'}u_{b',y'}^*u_{b,y})\right)_{a,a',b,b'\in \mathbb N_n}, \ \ \
x,x',y,y'\in \mathbb N_n.
\end{equation}
\end{itemize}
\end{theorem}

We will say that the quantum graphs ${\mathcal U}_G$, ${\mathcal U}_H$ are qc-isomorphic (q-isomorphic, loc-isomorphic)  and write ${\mathcal U}_G\simeq_{\rm qc} {\mathcal U}_H$ (resp. ${\mathcal U}_G\simeq_{\rm q} {\mathcal U}_H$, ${\mathcal U}_G\simeq_{\rm loc} {\mathcal U}_H$) if there exists a perfect quantum commuting (resp. quantum, local) strategy for the isomorphism game $\text{Iso}(\mathcal U_{G},\mathcal U_{H})$. Note that by \cite[Proposition 7.8]{bhtt}, ${\mathcal U}_G$ and ${\mathcal U}_H$ are  loc-isomorphic if and only if $G$ and $H$ are isomorphic. 

\begin{proposition}\label{UKn}
 $\mathcal U_{K_3}$ admits  nonlocal symmetry.   
\end{proposition}
\begin{proof}
For the generators $p_{x,a}$, $a,x =1,2,3$, $u_1, u_2$ of $C(\mathrm{Qut}(\mathcal U_{K_3}))$ from
Theorem~\ref{generator}, define
\begin{align*}
 \phi(u_1)& = \phi(u_2) = V=\begin{pmatrix}
   0&0&1\\1&0&0\\0&1&0
  \end{pmatrix},
 \\
 \phi(p_{1,1})& = \phi(p_{2,2}) = \phi(p_{3,3}) = P_1=
  \begin{pmatrix}
   1&0&0\\0&0&0\\0&0&0
  \end{pmatrix},
 \\
 \phi(p_{1,3})& = \phi(p_{2,1}) = \phi(p_{3,2}) = P_2 =\begin{pmatrix}
   0&0&0\\0&1&0\\0&0&0
  \end{pmatrix},
 \\
 \phi(p_{1,2})& = \phi(p_{2,3}) = \phi(p_{3,1}) = P_3 =  \begin{pmatrix}
   0&0&0\\0&0&0\\0&0&1
  \end{pmatrix}.
\end{align*}
Then $\phi$ can be extended to a homomorphism
\[
\phi\colon C(\mathrm{Qut}(\mathcal U_{K_3})) \to M_3(\mathbb
C).
\]
Indeed, since $P_1,P_2,P_3$ are projections and $V$ is unitary, it is sufficient
to verify that
\begin{gather*}
 \sum_{a=1}^3 \phi(p_{x,a})=\sum_{x=1}^3 \phi(p_{x,a})= I,
 \\
 \phi(p_{1,a}) V \phi(p_{2,a}) =  \phi(p_{2,a}) V \phi(p_{3,a}) =  \phi(p_{1,a}) V^2
\phi(p_{3,a}) =0, \quad a=1,2,3.
\end{gather*}
But it is a straightforward verification that these relations hold. Moreover, since
\begin{gather*}
 P_jVP_k \ne 0 \iff (j,k) \in \{(1,3), (2,1), (3,2)\},
 \\
 P_jV^2P_k \ne 0 \iff (j,k) \in \{(1,2), (2,3), (3,1)\},
\end{gather*}
we see that
\begin{align*}
\phi(u_{1,1}^*u_{2,3}) &= \phi(p_{1,1})\phi(u_1)\phi(p_{2,3}) = P_1VP_3 = E_{13},
\\
\phi(u_{1,2}^*u_{2,1}) &= \phi(p_{1,2})\phi(u_1)\phi(p_{2,1}) = P_3VP_2 = E_{32},
\\
\phi(u_{1,3}^*u_{2,2}) &= \phi(p_{1,3})\phi(u_1)\phi(p_{2,2}) = P_2 V P_1 = E_{21},
\\
\phi(u_{2,1}^*u_{3,3}) &= \phi(p_{2,1})\phi(u_2)\phi(p_{3,3}) = P_2VP_1 =E_{21},
\\
\phi(u_{2,2}^*u_{3,1}) &= \phi(p_{2,2})\phi(u_2)\phi(p_{3,1}) = P_1VP_3 = E_{13},
\\
\phi(u_{2,3}^*u_{3,2}) &= \phi(p_{2,3})\phi(u_2)\phi(p_{3,2}) = P_3 V P_2 = E_{32},
\\
\phi(u_{1,1}^*u_{3,2}) &= \phi(p_{1,1})\phi(u_1u_2)\phi(p_{3,2}) = P_1V^2P_2 =
E_{12},
\\
\phi(u_{1,2}^*u_{3,3}) &= \phi(p_{1,2})\phi(u_1u_2)\phi(p_{3,3}) = P_3V^2P_1
=E_{31},
\\
\phi(u_{1,3}^*u_{3,1}) &= \phi(p_{1,3})\phi(u_1u_2)\phi(p_{3,1}) = P_2 V^2 P_3 =
E_{23},
\end{align*}
where $E_{xa}$, $a,x=1,2,3$, are the standard matrix units, and all other elements
$u_{a,x}^*u_{b,y}$, $b>a$, are mapped by $\phi$ to zero.

Define a linear subspace $\mathcal R \subset C(\mathrm{Qut}(\mathcal
U_{K_3}))$ as follows
\[
 \mathcal R = \text{span} \{u_{a,x}^*u_{b,y}u_{a',x'}^*u_{b,'y'},
a,b,a',b',x,y,x',y' =1,2,3\}.
\]

Our aim is to show that on $C(\mathrm{Qut}(\mathcal U_{K_3}))$
there exists a trace $\tau$ which is ``nonlocal'' in the following sense: there
is no commutative $C^*$-algebra $\mathcal A$ such that for all $x \in \mathcal
R$
\[
 \tau(x) = \tau_{\mathcal A}(\pi(x)),
\]
where $\pi\colon C(\mathrm{Qut}(\mathcal U_{K_3})) \to
\mathcal A$ is a homomorphism, and $\tau_{\mathcal A}\in \mathcal A^*$.

We let $\tau = \text{tr} \circ \phi$, where $\phi \colon C(\mathrm{Qut}(\mathcal U_{K_3})) \to M_3(\mathbb C)$ is as constructed
above, and $\text{tr}$ is the standard normalized trace on $M_3(\mathbb C)$.
Then we have
\begin{align*}
 \rho(u_{1,1}^*u_{2,3}u_{3,1}^*u_{2,2})& = \text{tr}(\phi(u_{1,1}^*u_{2,3})
\phi(u_{2,2}^*u_{3,1})^*)
 \\
 & = \text{tr}(E_{13}E_{31}) = \text{tr}(E_{11}) =\frac13.
\end{align*}

On the other hand, for any commutative $C^*$-algebra $\mathcal A$, a  $*$-homomorphism,
$\pi \colon C(\mathrm{Qut}(\mathcal U_{K_3})) \to \mathcal
A$, and $\tau_{\mathcal A}\in \mathcal A^*$, by the commutativity we
have
\begin{align*}
 \tau_{\mathcal A}(\pi(u_{1,1}^*u_{2,3}u_{3,1}^*u_{2,2}))&=\tau_{\mathcal
A}(\pi(p_{1,1}u_1p_{2,3}p_{3,1} u_2^* p_{2,2}))
 \\
 &=\tau_{\mathcal A}(\pi(p_{1,1})\pi(u_1)\pi(p_{2,3})\pi(p_{3,1})\pi(u_2^*)
\pi(p_{2,2}))
 \\
 &=\tau_{\mathcal A}(\pi(p_{1,1})\pi(p_{3,1})\pi(u_1)\pi(u_2^*)\pi(p_{2,3})
\pi(p_{2,2})) =0,
\end{align*}
since $p_{1,1}p_{3,1} = p_{2,3}p_{2,2} =0$ in $C(\mathrm{Qut}(\mathcal
U_{K_3}))$. 
\end{proof}

\begin{remark}
    It was proved in \cite{RobSch21} that the graph $K_n$ does not admit nonlocal symmetry only if $n\geq 5$, meaning that in this case there is  no perfect NS ${\rm qc}$-strategy that cannot be produced classically, i.e. not in the loc-class. This shows that QNS strategies provide more  nonlocal ``symmetries''.  
 \end{remark}

\begin{remark}
We note that the representation $\phi$ from the proof of Proposition \ref{UKn} comes from the bi-unitary 
$\tilde{\mathcal U} = (\tilde u_{a,x})_{a,x=1}^3$  
whose entries are $3\times 3$
matrices and given by 
\begin{gather*}
\tilde u_{1,1} = E_{11}, \quad \tilde u_{1,2}= E_{23}, \quad \tilde u_{1,3} =
E_{32},
 \\
 \tilde u_{2,1} = E_{22}, \quad \tilde u_{2,2} = E_{31}, \quad \tilde u_{2,3} =
E_{13},
\\
\tilde u_{3,1} = E_{33}, \quad \tilde u_{3,2} = E_{12}, \quad \tilde u_{3,3}=
E_{21},
\end{gather*}
$\phi(u_{a,x}^*u_{b,y}) = \tilde u_{a,x}^* \tilde u_{b,y}$, $a,b,x,y =1,2,3$.
It is easy to see that $\phi(u_{a,x}^*u_{b,y}) = \tilde u_{a,x}^* \tilde u_{b,y}$, $a,b,x,y =1,2,3$.
Also notice that for such a choice of $(\tilde u_{a,x})_{a,x=1}^3$, the matrix
$\tilde {\mathcal U}$ is a $9\times 9$ classical permutation matrix. 
\end{remark}

\begin{remark}
 The bi-unitary matrix $\tilde {\mathcal U}$ from the remark  above can be
obtained as a result of a slight modification of the Latin square construction
discussed in \cite[Section 4.1]{RobSch21}.

 Let $A,B,C,D$ be flat unitary (complex Hadamard)  $n\times n$ matrices, that  is unitary matrices whose entries all have the same modulus (necessarily
equal to $\frac{1}{\sqrt{n}}$  when the matrices are $n\times n$).
Define
 \begin{align*}
  \psi_{jk} = {\sqrt{n}}\begin{pmatrix}a_{1j} b_{1k}\\ \vdots \\ a_{nj}
b_{nk}\end{pmatrix}, \quad
  \phi_{jk} ={\sqrt{n}}\begin{pmatrix}c_{1j} d_{1k}\\ \vdots \\ c_{nj}
d_{nk}\end{pmatrix}, \qquad j,k =1,\dots, n,
 \end{align*}
 where $a_{ij}$, $b_{ij}$, $c_{ij}$, $d_{ij}$ are the entries of $A$, $B$, $C$, $D$ respectively.
Then, since the matrices $A$ and $B$ are flat unitary, their entries have
absolute value $1/{\sqrt{n}}$, and
\begin{align*}
\psi_{jk}^* \psi_{jl} = n \sum_{i=1}^n \bar b_{ik} \bar a_{ij} a_{ij} b_{il} =
\sum_{i=1}^n \bar b_{ik}  b_{il} =\delta_{k,l}, \quad j,k,l=1,\dots, n,
\end{align*}
and similarly $\psi_{kj}^* \psi_{lj} =\delta_{k,l}$.
Also, since $\psi_{jk}\psi_{jk}^* = n(a_{lj} b_{lk} \bar b_{mk}\bar
a_{mj})_{l,m=1}^n$, we have
\begin{align*}
 \sum_{j=1}^n \psi_{jk}\psi_{jk}^*& = \Bigl(n\sum_{j=1}^n a_{lj} b_{lk} \bar
b_{mk}\bar a_{mj} \Bigr)_{l,m=1}^n
 \\
 &= \Bigl( nb_{lk} \bar b_{mk}\sum_{j=1}^n a_{lj}\bar a_{mj} \Bigr)_{l,m=1}^n
 =(\delta_{l,m})_{l,m=1}^n = I.
\\
 \sum_{j=1}^n \psi_{kj}\psi_{kj}^*& = \Bigl(n\sum_{j=1}^n a_{lk} b_{lj} \bar
b_{mj}\bar a_{mk} \Bigr)_{l,m=1}^n
 \\
 &= \Bigl(n a_{lk} \bar a_{mk}\sum_{j=1}^n  b_{lj} \bar b_{mj} \Bigr)_{l,m=1}^n
 =(\delta_{l,m})_{l,m=1}^n = I.
\end{align*}
The same relations hold obviously for $(\phi_{jk})_{j,k=1}^n$.

Now define the block martix
\[
\mathcal W = (w_{jk})_{j,k=1}^n, \quad w_{jk} = \psi_{jk}\phi_{jk}^*.
\]
We claim that $\mathcal W$ is bi-unitary. Indeed,
\begin{align*}
\sum_{i=1}^n w_{ij}^* w_{ik}& = \sum_{i=1}^n \phi_{ij}\psi_{ij}^*
\psi_{ik}\phi_{ik}^* = \delta_{jk} \sum_{i=1}^n \phi_{ij}\phi_{ij}^*
=\delta_{jk} I,
\\
\sum_{i=1}^n w_{ji}^* w_{ki}& = \sum_{i=1}^n \phi_{ji}\psi_{ji}^*
\psi_{ki}\phi_{ki}^* = \delta_{jk} \sum_{i=1}^n \phi_{ji}\phi_{ji}^*
=\delta_{jk} I.
\end{align*}

To obtain the bi-unitary matrix $\mathcal  U$ from the previous Remark,
take
\[
 A = B = C = \frac1{\sqrt{3}}
 \begin{pmatrix}
  1&1&1\\1&\epsilon&\epsilon^2\\1&\epsilon^2&\epsilon
 \end{pmatrix},
 \quad
 D= \begin{pmatrix}
  1&1&1\\1&\epsilon^2&\epsilon\\1&\epsilon&\epsilon^2
 \end{pmatrix},
\]
and orthonormal basis
\[
 e_1=\frac1{\sqrt{3}}\begin{pmatrix}1\\1\\1\end{pmatrix}, \quad
 e_2 = \frac1{\sqrt{3}}\begin{pmatrix}1\\ \epsilon \\\epsilon^2\end{pmatrix},
\quad
 e_3 = \frac1{\sqrt{3}}\begin{pmatrix}1 \\ \epsilon^2 \\ \epsilon\end{pmatrix}.
\]
where $\epsilon=e^{2\pi i/3}$. 
Then
\begin{gather*}
 \psi_{11} = \psi_{23} = \psi_{32} = \phi_{11} = \phi_{22} = \phi_{33} = e_1,
 \\
 \psi_{12} = \psi_{21} = \psi_{33} = \phi_{13} = \phi_{21} = \phi_{32} = e_2,
 \\
 \psi_{13} = \psi_{22} = \psi_{31} = \phi_{12} = \phi_{23} = \phi_{31} =e_3,
\end{gather*}
and therefore,
\begin{gather*}
u_{1,1} = e_1  e_1^* = E_{11}, \quad u_{1,2} = e_2 e_3^* = E_{23}, \quad u_{1,3} =
e_3 e_2^* = E_{32},
\\
u_{2,1} = e_2e_2^* = E_{22}, \quad u_{2,2} = e_3e_1^* = E_{31},\quad u_{2,3} =
e_1e_3^* = E_{13},
\\
u_{3,1} = e_3e_3^* = E_{33}, \quad u_{3,2} = e_1e_2^* = E_{12}, \quad u_{3,3} =
e_2e_1^* = E_{21}.
\end{gather*}
\end{remark}

\subsection{Decomposition into regular graphs}\label{dec}

Let $G_1$ and $G_2$ be simple undirected finite graphs. We assume that $G_1$
and $G_2$ have the same number of vertices $n$. We also write $V_j(G)$ for the
set of vertices of $G$ of degree (valence) $j$.

By Theorem \ref{th_alg}, the quantum graphs ${\mathcal  U}_{G_1}$ and ${\mathcal U}_{G_2}$ are qc-isomorphic if there exist a tracial $C^*$-algebra $\mathcal A$ and  a 
bi-unitary matrix $U=(u_{a,x})_{a\in V(G_2),x\in V(G_1)}$ which  satisfies  $u_{a,x}^*u_{b,y}\in\mathcal A$, $x,y\in V(G_1)$, $a,b\in V(G_2)$,  and the following orthogonality conditions
 \begin{align}
  \label{eq:orth}
  u_{a,x}u^*_{b,y}= 0, \quad \text{if $a \sim b$ in $G_2$, $x\not\sim y$ in $G_1$
or $a \not\sim b$ in $G_2$, $x\sim y$ in $G_1$},
 \end{align}
here we assume that $a\not \sim a$, $x \not \sim x$ for any  $a\in V(G_2)$, $x\in V(G_1)$.

The following statements are similar to the ones obtained in \cite{mfthesis} in the context of  quantum automorphism of graphs. 

\begin{lemma}\label{lm:orth}
 Let a bi-unitary block matrix $U$ satisfies the orthogonality
conditions~\eqref{eq:orth}. If $x\in V_j(G_1)$, $a\in V_k(G_2)$, $j\ne k$, then
$u_{a,x}=0$.
\end{lemma}

\begin{proof}
 Let $n=|V(G_1)|=|V(G_2)|$. Take $x\in V_j(G_1)$, $a\in V_k(G_2)$, $j\ne k$.
 Since $U$ is unitary, for any $y\in V(G_1)$, we have
 \[
  \sum_{b\in V(G_2)} u_{b,y}^* u_{b,y} =I.
 \]
Then for any $y\sim x$ in $G_1$, from $u_{a,x}u^*_{b,y}=0$, $a\not\sim b$ in $G_2$, we get
\[
 u_{a,x} = u_{a,x}\sum_{b\in V(G_2)} u_{b,y}^* u_{b,y}
 = u_{a,x}\sum_{b\sim a} u_{b,y}^* u_{b,y}.
\]
As $x \in V_j(G_1)$, taking sum over all $y\sim x$ we obtain
\[
 j u_{a,x} = u_{a,x}\sum_{y\sim x}\sum_{b\sim a} u_{b,y}^* u_{b,y}.
\]
Similarly, since $U$ is bi-unitary, for any $b\in V(G_2)$, we have
\[
 \sum_{y\in V(G_1)}u_{b,y}^* u_{b,y} =I,
\]
and since $a \in V_k(G_2)$, we get
\[
  k u_{a,x} = u_{a,x}\sum_{y\sim x}\sum_{b\sim a} u_{b,y}^* u_{b,y},
\]
which implies $ju_{a,x} = ku_{a,x}$ and hence $u_{a,x}=0$.
\end{proof}

\begin{proposition}\label{th:degree}
 Assume that there exists a bi-unitary matrix $U$ which satisfies the  orthogonality
conditions \eqref{eq:orth}.
Then $|V_k(G_1)| = |V_k(G_2)|$, $k\in \mathbb N_{n-1}$.
\end{proposition}

\begin{proof}
 Let $x \in V_j(G_1)$. Since $U$ is unitary, applying Lemma \ref{lm:orth}, we obtain
 \[
  I= \sum_{a\in V(G_2)} u_{a,x}^* u_{a,x} =\sum_{a\in V_j(G_2)} u_{a,x}^* u_{a,x},
 \]
 which implies, in particular, that $V_j(G_2)$ is non-empty. Then
 \[
  \sum_{x\in V_j(G_1)}\sum_{a\in V_j(G_2)} u_{a,x}^* u_{a,x} =|V_j(G_1)| I.
 \]
Similarly, since $U$ is bi-unitary, we see that for $a \in V_j(G_2)$,
\[
 I= \sum_{x\in V(G_1)} u_{a,x}^* u_{a,x} =\sum_{x\in V_j(G_1)} u_{a,x}^* u_{a,x},
\]
which yields
\[
  \sum_{x\in V_j(G_1)}\sum_{a\in V_j(G_2)} u_{a,x}^* u_{a,x} =|V_j(G_2)| I.
\qedhere
\]
\end{proof}

Recall that for $S\subset V(G)$ a subgraph of $G$ induced by $S$ is defined as a graph whose set of vertices is $S$, and two vertices in $S$ are connected by edge iff they are connected by edge in $G$.

\begin{lemma}\label{lm:diag}
Let $G_1^k$ and $G_2^k$, $k\in \mathbb N_{n-1}$, $n=|V(G_1)|=|V(G_2)|$, be subgraphs of $G_1,G_2$ induced respectively by $V_k(G_1)$ and $V_k(G_2)$.
Order the vertices of $G_1$ and $G_2$ by their degree. Then a bi-unitary $U$ satisfying \eqref{eq:orth} has a block diagonal form, $U=\diag(U_1,\dots,U_{n-1})$,  $\dim U_k = |V_k(G_1)| =|V_k(G_2)|$, where $U_k$ is a bi-unitary which satisfies \eqref{eq:orth} for graphs $G_1^k$, $G_2^k$ (if $|V_k(G_1)| =|V_k(G_2)|=0$, then the corresponding block is empty).
\end{lemma}

\begin{proof}
The block diagonal form of $U$ follows from Lemma~\ref{lm:orth} and Proposition~\ref{th:degree}. Since $U$ is bi-unitary, each diagonal block is obviously bi-unitary. The orthogonality conditions follow since $G_1^k,G_2^k$ are induced subgraphs. 
\end{proof}

Recall that a graph is called regular if each its vertex has the same number of neighbors (degree). 

\begin{theorem}\label{th:core_graph}
 Let $U$ be a bi-unitary matrix which satisfies the orthogonality conditions
\eqref{eq:orth}.

Then there exist a finite set $\Lambda$,  a family $(G_1^\lambda)_{\lambda\in
\Lambda}$ of regular induced subgraphs of $G_1$ and a family
$(G_2^\lambda)_{\lambda \in \Lambda}$, of regular induced subgraphs of $G_2$
such that:

(i) Degrees of vertices in $G_1^\lambda$ and $G_2^\lambda$ coincide for all
$\lambda \in \Lambda$;

(ii) $|V(G_1^\lambda)| =|V(G_2^\lambda)|$, $\lambda \in \Lambda$;

(iii) after an appropriate ordering of vertices of $G_1$, $G_2$ (and the
corresponding rows and columns of $U$), the matrix $U$ takes the form
\[
 U = \diag(U_\lambda)_{\lambda \in \Lambda},
\]
where each $U_\lambda$, $\lambda \in \Lambda$ is a bi-unitary block matrix of
size $n_\lambda=|V(G_1^\lambda)| = |V(G_2^\lambda)|$ which satisfies the
orthogonality conditions imposed by the pair of graphs $G_1^\lambda$,
$G_2^\lambda$.
\end{theorem}

\begin{proof}
Order the vertices of $G_1,G_2$ by their degree (at this step, the order of vertices of the same degree can be taken arbitrarily). Then by Lemma~\ref{lm:diag}, construct induced subgraphs $G_1^{k_1}, G_2^{k_1}$, $k_1\in\mathbb N_{n-1}$, and decompose $U$ into a direct sum $U=\diag(U_{k_1})_{k_1 = 1}^{n-1}$ of bi-unitary matrices corresponding to pairs $G_1^{k_1},G_2^{k_1}$.
If for some $k_1$ the graph $G_1^{k_1}$ is not regular (and therefore, by Proposition~\ref{th:degree} $G_2^{k_1}$ is not regular, too), repeat the same procedure: order vertices inside   $G_1^{k_1}$ and $G_2^{k_1}$ by their degree in these subgraphs, consider induced subgraphs $G_1^{k_1k_2}, G_2^{k_1k_2}$, $k_2\in\mathbb N_{n_1-1}$, $n_1 = |V_{k_1}(G_1)|=|V_{k_1}(G_2)|$, and further decompose $U_{k_1}= \diag(U_{k_1,1},\dots,U_{k_1,n_1-1})$; do it recursively for all non-regular subgraphs. As a result, we obtain a family of regular induced subgraphs $G_1^\lambda$, $G_2^\lambda$, where $\lambda$ belongs to some set $\Lambda$ of multiindices of varied length, and a decomposition $U = \diag(U_\lambda)_{\lambda \in \Lambda}$ such that (i), (iii) holds by the construction, and (ii) follows from  Proposition~\ref{th:degree}.
 \end{proof}

\begin{corollary} Let $G_1$ and $G_2$ be simple graphs such that $\mathcal U_{G_1}\simeq_{\rm t} \mathcal U_{G_2}$, ${\rm t}\in\{\rm loc, q, qc\}$. Then for each $\lambda\in \Lambda$ and the regular induced subgraphs $G_1^\lambda$ and  $G_2^\lambda$, we have ${\mathcal U}_{G_1^\lambda}\simeq_{\rm t} {\mathcal U}_{G_2^\lambda}$. 
\end{corollary}

\subsection{Nonlocal symmetries for higher order graphs.} \label{nonloc}

In this section we show that all quantum graphs ${\mathcal U}_G$  admit nonlocal symmetries if and only if $|V(G)|\geq 3$. This stands in contrast to the situation for nonlocal symmetries of the underlying classical graphs, where the theory is considerably richer, see \cite{RobSch21}. In particular, it was proved in \cite{RobSch21} that the only graphs on five vertices that does not admit nonlocal symmetries are $K_5$ and its complement $\bar{K}_5$.

We first define the class of synchronous correlations and recall a characterization which will be used in the proof of the main theorem.

Let $\mathbb F(n,k)$ be the group that is the free product of $n$ copies of the
cyclic group of order $k$ and let $C^*(\mathbb F(n,k))$ be the $C^*$-algebra of
$\mathbb F(n,k)$. It is generated by $n$ unitaries $u_1,\ldots, u_n$, each of
order $k$, i.e. $u_i^k=1$, $i\in\mathbb N_n$. Decomposing $u_i$ in terms of
spectral projections $u_i=\sum_{a=0}^{k-1} \omega^ae_{i,a}$, where
$\omega=e^{2\pi i/k}$, we obtain a system of PVMs $\{e_{i,a}: a=0,\ldots,
k-1\}$, i.e. projections $e_{i,a}$ such that $\sum_{a=0}^{k-1} e_{i,a}=1$ for all $i$, which is also a universal
family of $n$ $k$-PVM's, meaning that for any family of PVMs $\{E_{x,a}:
a=0,\ldots,k-1\}\subset \mathcal B(H)$, $x\in\mathbb N_n$, there exists a
$*$-homomorphism of $C^*(\mathbb F(n,k))\to \mathcal B(H)$ such that
$e_{x,a}\mapsto E_{x,a}$.

Let $X=Y$, $A=B$ be finite sets. Recall that a correlation $\{p(a,b|x,y)\}_{x,y\in X, a,b\in
A}$ is synchronous if $p(a,b|x,x)=0$ whenever $a=b$. We write $\mathcal C_{\rm t}^s(n,k)$ for the t-class of synchronous correlations with $|X|=|Y|=n$, $|A|=|B|=k$, where ${\rm t}\in\{{\rm loc, q, qc}\}$.

The following characterization of synchronous correlations in terms of traces is
from \cite{psstw}.
\begin{theorem}\label{synch}
We have $p\in \mathcal C_{\rm qc}^s(n,k)$ if and only if there is a family of
$n$ $k$-PVM's $\{E_{x,a}: 1\leq x\leq n, 0\leq a\leq k-1\}$ in a $C^*$-algebra $\mathcal A$
with a trace $\tau$ such that $p(a,b|x,y)=\tau(E_{x,a}E_{y,b}).$
Moreover,
\begin{itemize}
    \item $p\in \mathcal C_{\rm loc}^s(n,k)$ if and only if $\mathcal A$ can be
taken to be abelian,
    \item $p\in \mathcal C_{\rm q}^s(n,k)$ if and only if $\mathcal A$ can be
taken to be finite-dimensional.
\end{itemize}
\end{theorem}

An example showing a separation between local $\mathcal C_{\rm
loc}^s(3,2)$ and  quantum synchronous correlations $\mathcal C_{\rm q}^s(3,2)$ is
given in \cite[Example 4.4]{MPTW}. In particular, together with Theorem
\ref{synch}, it provides a trace $\tilde\tau$ on $C^*(\mathbb F(3,2))$ such that
there is no representation
$\pi$ of $C^*(\mathbb F(3,2))$ into an abelian $C^*$-algebra $\mathcal A$ and a
linear functional $f\in \mathcal A^*$ such that
$\tilde\tau(e_{x,a}e_{y,b})=f(\pi(e_{x,a}e_{y,b}))$ for all $1\leq x,y\leq 3$,
$a,b\in\{0,1\}$. Write $v_x$, $1\leq x\leq 3$,
for the generators of $C^*(\mathbb F(3,2))$ and let $\mathcal R=\text{span}\{1,v_x,
v_xv_y, 1\leq x,y\leq 3 \}$. As $p_{x,a}=\frac{1}{2}(1+(-1)^av_x)$, we have that
the restriction $\tilde\tau|_{\mathcal R}$ is not equal to the restriction
$f\circ \pi|_{\mathcal R}$ for any $\pi$ and $f$ as above.

We are now ready to prove the main theorem of this section. 

\begin{theorem}
Let $G$ be a graph with $|V(G)|\ge3$. Then $\mathcal U_G$ admits a nonlocal
symmetry.
\end{theorem}

\begin{proof}
Assume first that $|V(G)|=3$. 
In Proposition~\ref{UKn} we have already proved the statement for the
full graph $K_3$. The three remaining graphs are $\bullet - \bullet - \bullet$,
$\bullet\ \ \ {\bullet - \bullet} = K_1\cup K_2$,  and $\bullet \ \
\bullet\ \ \bullet = K_1\cup K_1\cup K_1$.
As in the case of $K_3$, we prove that for each  such graph $G$, there exists a trace $\tau_G$ on $C(\mathrm{Qut}(\mathcal U_G))$, which cannot be factored through a commutative $C^*$-algebra $\mathcal
A$ in such a way that
\[
 \tau(x) = \tau_{\mathcal A}(\pi(x)),
\]
where $\tau_{\mathcal A}\in \mathcal A^*$,
$\pi\colon C(\mathrm{Qut}(\mathcal U_G)) \to
\mathcal A$ is a homomorphism, and
\[
 x \in\mathcal R = \text{span} \{u_{a,x}^*u_{b,y}u_{a',x'}^*u_{b,'y'},
a,b,a',b',x,y,x',y' =1,2,3\}.
\]

Consider the graph $G = {\stackrel{1}{\bullet}\ \ \ \stackrel{2}{\bullet}\! -
\stackrel{3}{\bullet}} = K_1\cup K_2$.
Let $U=(u_{a,x})_{a,x\in V(G)}$ be the  bi-unitary  such that
$u_{a,x}^*u_{b,y}$, $x,y,a,b\in V(G)$ generate  $C(\mathrm{Qut}(\mathcal
U_{G}))$.  By Lemma~\ref{lm:orth}, $u_{1,1}$ is unitary,
$u_{1,2}=u_{2,1}=u_{1,3}=u_{3,1}=0$ and $(u_{a,x})_{a,x=2}^3$ is a bi-unitary
such that $u_{a,x}^*u_{b,y}$, $x,y,a,b=2,3$, generate a subalgebra isomorphic to
$C(\mathrm{Qut}(\mathcal
U_{K_2}))$. By Theorem~\ref{prop:aiso} and the proof of
Proposition~\ref{graph_K2}, $u_{a,x}$, $a,x=2,3$, are partial isometries, so
$u_{2,2}^*u_{2,2}=u_{3,3}^*u_{3,3}:=p$, $u_{2,3}^*u_{2,3}=u_{3,2}^*u_{3,2}=1-p$
are projections. Therefore, if $\pi\colon C(\mathrm{Qut}(\mathcal
U_{G}))\to \mathcal A$ is a $*$-homomorphism to a commutative
$C^*$-algebra $\mathcal A$, then
\begin{align*}
\pi(u_{2,2}^*u_{1,1}u_{3,2}^*u_{1,1})&
=\pi(u_{2,2}^*u_{2,2})\pi(u_{2,2}^*u_{1,1} )
\pi(u_{3,2}^*u_{3,2})\pi(u_{3,2}^*u_{1,1})
\\
&=\pi(p)\pi(u_{2,2}^*u_{1,1} )
\pi(1-p)\pi(u_{3,2}^*u_{1,1})
\\
&=\pi(u_{2,2}^*u_{1,1} )\pi(p)
\pi(1-p)\pi(u_{3,2}^*u_{1,1})=0.
\end{align*}
On the other hand, for the bi-unitary $U=(u_{a,x})_{a,x=1}^3$ given by
\begin{gather*}
	u_{1,1} = \begin{pmatrix}
		0&1\\1&0
	\end{pmatrix},
	\quad u_{1,2} = u_{1,3} = u_{2,1} = u_{3,1} =0,
	\\
	u_{2,2} = u_{3,3} = \begin{pmatrix}
		1&0\\0&0
	\end{pmatrix},
	\quad
	u_{2,3} = u_{3,2} = \begin{pmatrix}
		0&0\\0&1
	\end{pmatrix},
\end{gather*}
which can be easily verified to determine a two-dimensional representation of $C(\mathrm{Qut}(\mathcal
U_{G}))$, we have that 
\[
\text{tr}( u_{2,2}^*u_{1,1}u_{3,2}^*u_{1,1}) = \text{tr}(p) = \frac12,
\]
which proves the existence of a nonlocal symmetry for this graph.

For the bi-unitary matrix  $U=(u_{a,x})_{a,x\in V(G)}$ related to the graph
$G=\ \stackrel{2}{\bullet}-\stackrel{1}{\bullet}-\stackrel{3}{\bullet}$ we have
by
Lemma~\ref{lm:orth}, as above, that $u_{11}$ is unitary,
$u_{1,2}=u_{2,1}=u_{1,3}=u_{3,1}=0$, and $u_{ax}$, $a,x\in \{2,3\}$ form a
bi-unitary matrix, but unlike above, there are no non-trivial orthogonality
conditions.
Consider $C^*$-algebra $C^*(\mathbb
F(3,2))/\mathcal J$, where $\mathcal J$ is a $C^*$-ideal generated by
$u_1+u_2+u_3$, where  $u_i$, $i=1,2,3$, are the canonical selfadjoint unitary
generators of
$C^*(\mathbb F(3,2))$. This algebra has a unique, up to unitary equivalence,
irreducible representation $\pi$,
\begin{equation} \label{eq:mercedes-benz}
 \pi(u_1) =
 \begin{pmatrix}
1&0\\0&-1
 \end{pmatrix}, \quad
 \pi(u_2) =
 \begin{pmatrix}
-1/2&\sqrt{3}/2 \\\sqrt{3}/2&1/2
 \end{pmatrix}, \quad
 \pi(u_3) =
 \begin{pmatrix}
-1/2&-\sqrt{3}/2 \\-\sqrt{3}/2&1/2
 \end{pmatrix},
\end{equation}
and thus is simple. Using directly verified relation
$u_1u_2=u_2u_3$,
it is a routine verification that
\[
 \mathcal U = (\pi(u_{a,x}))_{a,x=1}^3 =
 \frac1{\sqrt{2}}
 \begin{pmatrix}
  \sqrt{2}I&0&\\
  0&\pi(u_1)&\pi(u_2)\\
  0&\pi(u_2)&-\pi(u_3)
 \end{pmatrix}
\]
is a bi-unitary matrix.

For the corresponding representation $\pi$ of $C(\mathrm {Qut}(\mathcal U_G))$,
we have, in particular,
\begin{align*}
 \pi(u_{1,1}^*u_{2,2}) = \pi(u_1),
 \quad
 \pi(u_{1,1}^*u_{2,3}) = \pi(u_2),
 \quad
\pi(u_{1,1}^*u_{3,3})=  -\pi(u_3).
\end{align*}
By Theorem~\ref{th_alg}, the trace $\tau=\mathrm{tr}\circ\pi$ gives rise to a
perfect quantum strategy on $C(\mathrm {Qut}(\mathcal U_G))$ by
\[
\Gamma(\epsilon_{x,x'} \otimes \epsilon_{y,y'})
= \left(\tau(u_{a,x}^*u_{a',x'}u_{b',y'}^*u_{b,y})\right)_{a,a',b,b'},\quad
x,x',y,y',a,a',b,b'\in \mathbb N_3.
\]

Assuming this strategy being local, there exists a commutative $C^*$-algebra
$\mathcal A$, a state $\gamma \in \mathcal A^*$, and a homomorphism $\rho\colon
C(\mathrm {Qut}(\mathcal U_G)) \to \mathcal A$,
such that for each
\[
 s\in \mathcal S = \mathrm{span}\{u_{a,x}^*u_{a',x'}u_{b,y}^*u_{b',y'}\mid
x,x',y,y',a,a',b,b' \in \mathbb N_n\}
\]
we have $\tau(s) = \gamma(\rho(s))$. Here we may assume that $\gamma$ is
faithful.
Then we have
\begin{align*}
&\gamma(\rho((u_{1,1}^*u_{2,2}-u_{2,2}^*u_{1,1})^*(u_{1,1}^*u_{2,2}-u_{2,2}^*u_{
1 ,1}))
\\
&\qquad=
\mathrm{tr}\,\pi((u_{1,1}^*u_{2,2}-u_{2,2}^*u_{1,1})^*(u_{1,1}^*u_{2,2}-u_{2,2}^
* u_ {1,1}))
\\
&\qquad = \mathrm {tr} (\pi(u_1)-\pi(u_1^*))^*(\pi(u_1)-\pi(u_1^*))=0,
\end{align*}
so that $\rho(u_{1,1}^*u_{2,2})$ is self-adjoint. Here we used the fact that
\[
(u_{1,1}^*u_{2,2}-u_{2,2}^*u_{1,1})^*(u_{1,1}^*u_{2,2}-u_{2,2}^*u_{
1 ,1}) \in \mathcal S.
\]
Notice that $ 1=(u_{1,1}^*u_{1,1})^2\in  \mathcal S$,
then, as $\|u_{1,1}^*u_{2,2}\|\le 1$, we have $0\le 1-
u_{1,1}^*u_{2,2}u_{2,2}^* u_{1,1}\in \mathcal S$, and
\begin{align*}
 \gamma(\rho(1-
u_{1,1}^*u_{2,2}u_{2,2}^* u_{1,1}))& = \mathrm{tr}\, \pi(1-
u_{1,1}^*u_{2,2}u_{2,2}^* u_{1,1})) =\mathrm{tr}\, \pi(1 - u_1u_1)=0,
\\
\gamma(\rho(1-
u_{2,2}^*u_{1,1}u_{1,1}^* u_{2,2}))& = \mathrm{tr}\, \pi(1-
u_{2,2}^*u_{1,1}u_{1,1}^* u_{2,2})) =\mathrm{tr}\, \pi(1 - u_1u_1)=0,
\end{align*}
therefore, $\rho(u_{1,1}^*u_{2,2})$ is also unitary. The same way we see that
$\rho(u_{1,1}^*u_{2,3})$ and $\rho(u_{1,1}^*u_{3,3})$ are self-adjoint
unitaries. Similarly, from 
\begin{align*}
& \gamma(\rho((u_{1,1}^*u_{2,2} + u_{1,1}^*u_{2,3} -
u_{1,1}^*u_{3,3})^*(u_{1,1}^*u_{2,2} + u_{1,1}^*u_{2,3} - u_{1,1}^*u_{3,3})))
\\&\qquad
=\mathrm {tr}\,\pi((u_1+u_2+u_3)^*(u_1+u_2+u_3))=0
\end{align*}
we conclude that $\rho(u_{1,1}^*u_{2,2} + u_{1,1}^*u_{2,3} -
u_{1,1}^*u_{3,3})=0$.

This way, the mapping $u_1\mapsto \rho(u_{1,1}^*u_{2,2}), u_2\mapsto
\rho(u_{1,1}^*u_{2,3}), u_3\mapsto -\rho( u_{1,1}^*u_{3,3})$  gives rise to a non-trivial
$*$-homomorphism of $C^*(\mathbb F(3,2))/\mathcal J \to \mathcal A$, which is
impossible as  $C^*(\mathbb F(3,2))/\mathcal J \simeq M_2(\mathbb C)$, and
$\mathcal A$ is commutative.

For the graph  $G=K_1\cup K_1\cup K_1$, the set of orthogonality conditions
\eqref{eq:orth} is empty, therefore, any biunitary $3\times3$ matrix determines a
representation of $C(\mathrm{Qut}(\mathcal U_G))$.
Then the arguments used above for the graph
${\bullet}-{\bullet}-{\bullet}$ give an
example of nonlocal symmetry.

Assume now  $n:=|V(G)|\geq 4$.
Let $v_i$, $i \in \mathbb N$, be unitary operators on a Hilbert space $H$. Then
$U=\text{diag}(v_1,\ldots, v_n)\in M_n(B(H))$ is a bi-unitary operator which
generates a representation of $C(\mathrm{Qut}(\mathcal U_{G}))$ by letting
\begin{equation}\label{eq:diag-rep}
\rho(u_{a,x}^*u_{a',x'})=\delta_{a,x}\delta_{a',x'}v_a^*v_{a'}.
\end{equation}
Indeed,
\[
 \rho(u_{a,x}^*u_{a',x'}u_{b,y}^*u_{b',y'})=
\delta_{a,x}\delta_{a',x'}v_a^*v_{a'}
\delta_{b,y}\delta_{b',y'}v_b^*v_{b'}
\]
is zero for $a'\sim b,x'\not\sim y$ or  $a'\not\sim b,x'\sim y$ since for
nonzero cases we have $a'=x', b=y$.
In particular,
we have $\rho(u_{a,x}^*u_{a,x})=\delta_{a,x}I$.

In \eqref{eq:diag-rep}, take $H=\mathbb C^2$ and $\pi$ as defined in \eqref{eq:mercedes-benz} and let
\[
 v_2 = v_1 \pi(u_1), \quad v_3= v_2 \pi(u_2) = v_1 \pi(u_1u_2), \quad v_4 =
v_3 \pi(u_3) = v_1 \pi(u_1u_2u_3),
\]
so that $v_1^*v_2=\pi(u_1)$, $v_2^*v_3=\pi(u_2)$, $v_3^*v_4=\pi(u_3)$. 
If we assume that the corresponding winning strategy for $C(\mathrm{Qut}(\mathcal U_{G}))$ is local, then the arguments similar to the above show that the mapping $s\colon u_a\mapsto \pi'(u_{a,a}^*u_{a+1,a+1})$, $a=1,2,3$,  gives rise to a
$*$-homomorphism  of $C^*(\mathbb F(3,2))/\mathcal J$
into a commutative $\mathcal A$ which leads to a contradiction.
\end{proof}


\begin{thebibliography}{MPTW23}

\bibitem[AMR{\etalchar{+}}19]{amrssv}
A.~Atserias, L.~Man\v{c}inska, D.E. Roberson, R.~S\'{a}mal, S.~Severini, and
  A.~Varvitsiotis, \emph{Quantum and non-signalling graph isomorphisms},
  Journal of Combinatorial Theory, Series B \textbf{136} (2019), 289--328.

\bibitem[Ban97]{banica1}
T.~Banica, \emph{Le groupe quantique compact libre ${U}(n)$}, Comm. Math. Phys.
  \textbf{190} (1997), no.~1, 143--172.

\bibitem[Ban99]{banica2}
T.~Banica, \emph{Symmetries of a generic coaction}, Math. Ann. \textbf{314}
  (1999), no.~4, 763--780.

\bibitem[Ban05]{banica}
T.~Banica, \emph{{Quantum automorphism groups of homogeneous graphs}}, Journal
of
  Functional Analysis \textbf{224} (2005), 243--280.

\bibitem[BBC07]{bbc}
Teodor Banica, Julien Bichon, and Beno\^it Collins, \emph{Quantum permutation
  groups: a survey}, Noncommutative harmonic analysis with applications to
  probability, Banach Center Publ., vol.~78, Polish Acad. Sci. Inst. Math.,
  Warsaw, 2007, pp.~13--34. \MR{2402345}

\bibitem[BHTT24]{bhtt}
M.~Brannan, S.J. Harris, I.G. Todorov, and L.~Turowska, \emph{{Quantum
  no-signalling bicorrelations}}, Adv. Math. \textbf{449} (2024), Paper No.
  109732, 81 pp.

\bibitem[DW16]{duan-winter}
R.~Duan and A.~Winter, \emph{No-signalling assisted zero-error capacity of
  quantum channels and an information theoretic interpretation of the
  {L}ov\'{a}sz number}, IEEE Trans. Inf. Theory \textbf{62} (2016), no.~2,
  891--914.

\bibitem[Ful06]{mfthesis}
M.~B. Fulton, \emph{{The quantum automorphism group and undirected trees}},
  Ph.D. thesis, Virginia Polytechnic Institute and State University, July 2006,
  pp.~1--74.

\bibitem[LMR20]{lmr}
M.~Lupini, L.~Man\v{c}inska, and D.E. Roberson, \emph{Nonlocal games and
  quantum permutation groups}, Journal of Functional Analysis \textbf{279}
  (2020), no.~5, 108592.

\bibitem[MPTW23]{MPTW}
L.~Man\v{c}inska, V.I. Paulsen, I.G. Todorov, and A.~Winter, \emph{Products of
  synchronous games}, Studia Math. \textbf{272} (2023), no.~3, 299--317.

\bibitem[PR21]{paulsen-rahaman}
V.I. Paulsen and M.~Rahaman, \emph{Bisynchronous games and factorizable maps},
  Ann. Henri Poincar\'{e} \textbf{22} (2021), no.~2, 593--614.

\bibitem[PSS{\etalchar{+}}16]{psstw}
Vern~I. Paulsen, Simone Severini, Daniel Stahlke, Ivan~G. Todorov, and Andreas
  Winter, \emph{Estimating quantum chromatic numbers}, J. Funct. Anal.
  \textbf{270} (2016), no.~6, 2188--2222. \MR{3460238}

\bibitem[RS21]{RobSch21}
David~E. Roberson and Simon Schmidt, \emph{{Quantum symmetry vs nonlocal
  symmetry}}, arXiv:2012.13328v2 [math.QA], 2021.

\bibitem[TT24]{tt-QNS}
I.G. Todorov and L.~Turowska, \emph{Quantum no-signalling correlations and
  non-local games}, Comm. Math. Phys. \textbf{405} (2024), no.~6, Paper No.
  141, 65 pp.

\bibitem[Voi17]{voigt}
C.~Voigt, \emph{{On the structure of quantum automorphism groups}}, J. Reine
  Angew. Math. \textbf{732} (2017), 255–273.

\bibitem[Wan98]{wang}
S.~Wang, \emph{Quantum symmetry groups of finite spaces}, Comm Math Phys
  \textbf{195} (1998), 195–211.

\bibitem[Web23]{moritz}
Moritz Weber, \emph{Quantum permutation matrices}, Complex Anal. Oper. Theory
  \textbf{17} (2023), no.~3, Paper No. 37, 26. \MR{4564553}

\end{thebibliography}
\newcommand{\etalchar}[1]{$^{#1}$}

\end{document}